%% file: notes.tex
\documentclass{amsart}

\usepackage[utf8]{inputenc}
\usepackage[T1]{fontenc}
\usepackage[hidelinks, unicode]{hyperref}
\usepackage{cleveref}
\usepackage{amsthm}
\usepackage{amssymb}
\usepackage{bm}
\usepackage{quiver}

\usepackage{import}

\usepackage{graphicx}

\usepackage{tikzit}
\input{tikzstyles.tikzstyles}
\usetikzlibrary{decorations.markings}

\usepackage{caption}
\usepackage{xparse}

\usepackage{nicematrix,booktabs}

\usepackage{stmaryrd}

\numberwithin{figure}{section}

\newtheorem{theorem}{Theorem}[section]
\newtheorem*{theorem*}{Theorem}

\newtheorem{lemma}[theorem]{Lemma}
\newtheorem*{lemma*}{Lemma}
\newtheorem{proposition}[theorem]{Proposition}
\newtheorem*{proposition*}{Proposition}

\newtheorem{conjecture}[theorem]{Conjecture}

\newtheorem{corollary}[theorem]{Corollary}
\newtheorem*{corollary*}{Corollary}

\theoremstyle{definition}
\newtheorem{remark}[theorem]{Remark}
\newtheorem{notation}[theorem]{Notation}
\newtheorem{example}[theorem]{Example}

\newtheorem*{question*}{Question}
\newtheorem{definition}[theorem]{Definition}

\newtheorem*{definition*}{Definition}

\newcommand{\Cat}{\mathcal{C}}

\newcommand{\dol}{\overline{\partial}}

\newcommand{\cA}{\mathcal{A}}

\newcommand{\Cinf}{C^\infty}

\newcommand{\cO}{\mathcal{O}}

\newcommand{\VL}{\mathrm{Vect}_\Lambda}
\newcommand{\VLone}{\mathrm{Vect}_{\Lambda_1}}
\newcommand{\VLtwo}{\mathrm{Vect}_{\Lambda_2}}
\newcommand{\CN}{\mathrm{C}_\Nu}

\newcommand{\GMod}{\mathfrak{gmod}}
\newcommand{\BraidFun}{\mathfrak{braidfun}}

\newcommand{\Braid}{\mathfrak{Braid}}
\newcommand{\Rib}{\mathfrak{Rib}}

\newcommand{\Modcat}{\mathfrak{Mod}}

\newcommand{\lra}{\longrightarrow}

\newcommand{\ra}{\rightarrow}

\newcommand{\Nu}{\mathcal{V}}

\newcommand{\PMod}[1]{\mathrm{PMod}(#1)}

\newcommand{\tPMod}[1]{\mathrm{PM\tilde{o}d}(#1)}

\newcommand{\PB}{\mathrm{PB}}

\protected\def\myphantom#1{\vphantom{#1}}

\newcommand{\mysecondleftidx}[3]{{\myphantom{#2}}#1#2#3}

\newcommand{\Mrpo}[2]{\mysecondleftidx{^1}{\overline{\mathcal{M}}}{_{0,{#1}}}(#2)}

\newcommand{\Mg}[2]{\mathcal{M}_{{#1},{#2}}}

\newcommand{\tMgb}[2]{\widetilde{\mathcal{M}}_{#1}^{#2}}

\newcommand{\oMg}[2]{\overline{\mathcal{M}}_{{#1},{#2}}}

\newcommand{\Mgrp}[3]{\overline{\mathcal{M}}_{{#1},{#2}}(#3)}
\newcommand{\Mgrpo}[3]{\mysecondleftidx{^1}{\overline{\mathcal{M}}}{_{{#1},{#2}}}(#3)}

\newcommand{\Cgrp}[3]{\overline{\mathcal{C}}_{{#1},{#2}}(#3)}

\newcommand{\Mgrb}[3]{\overline{\mathcal{M}}_{#1}^{#2}(#3)}
\newcommand{\Cgrb}[3]{\overline{\mathcal{C}}_{#1}^{#2}(#3)}

\newcommand{\Mgrbt}[4]{\overline{\mathcal{M}}_{#1}^{#2}(#3,#4)}

\newcommand{\Br}[1]{\mathrm{B}\mu_{#1}}

\newcommand{\End}[1]{\mathrm{End}(#1)}

\newcommand{\GLn}[2]{\mathrm{GL}_{#1}(#2)}

\newcommand{\Hom}[2]{\mathrm{Hom}(#1,#2)}

\newcommand{\Z}{\mathbb{Z}}

\newcommand{\C}{\mathbb{C}}

\newcommand{\id}{\mathrm{id}}

\newcommand{\ul}{\underline{\lambda}}

\newcommand{\ZD}{\mathrm{Z}_{D'}}
\newcommand{\HD}{\mathrm{H}_{D'}}

\newcommand{\CdR}{\Cat_{dR}}
\newcommand{\CdRs}{\Cat_{dR}^{ss}}
\newcommand{\CD}{\Cat_{D'}}
\newcommand{\CH}{\Cat_{\mathrm{H}}}

\newcommand{\hCdRs}{\hat{\Cat}_{dR}^{ss}}
\newcommand{\hCD}{\hat{\Cat}_{D'}}
\newcommand{\hCH}{\hat{\Cat}_{\mathrm{H}}}

\newcommand{\CdRz}{\Cat_{dR}^0}
\newcommand{\hCDz}{\hat{\Cat}_{D'}^0}
\newcommand{\hCHz}{\hat{\Cat}_{\mathrm{H}}^0}

\newcommand{\Lg}{\mathfrak{g}}
\newcommand{\sltwo}{\mathfrak{sl}_2}

\title{Semisimplicity of conformal blocks}
\author{Pierre Godfard}

\begin{document}

\begin{abstract}
  We prove that braid group representations associated to braided fusion categories
  and mapping class group representations associated to modular fusion categories are always semisimple.
  The proof relies on the theory of extensions in non-Abelian Hodge theory and on Ocneanu rigidity.
  By combining this with previous results on the existence of variations of Hodge structure,
  we further show that such a braid group or mapping class group representation
  preserves a non-degenerate Hermitian form and can be defined over some CM number field.
\end{abstract}

\maketitle


\section{Introduction}


For modular fusion categories constructed from Lie algebras, it is known that their associated mapping class group representations are unitary and,
hence, semisimple. While unitarity does not hold for general modular fusion categories, in this paper,
we prove that semisimplicity holds for all mapping class group representations associated to modular fusion categories.
The approach is axiomatic and based on non-Abelian Hodge theory and Ocneanu rigidity. We also address the cases
of ribbon fusion categories and braided fusion categories.


\subsection{Quantum representations and conformal blocks}\label{introquantumrep}


For each choice of a simple Lie algebra $\mathfrak{g}$ over $\C$ and of integer $\ell\geq 1$, called level,
the Reshetikhin-Turaev construction produces a collection of representations of mapping class groups.
More precisely, there is a finite set $\Lambda_\ell$ of integral dominant weights for $\mathfrak{g}$,
such that for each $g,n\geq 0$ and $\lambda_1,\dotsc,\lambda_n\in\Lambda_\ell$, a representation
\begin{equation*}
  \rho_g(\ul):\tPMod{S_g^n}\lra \GLn{d}{\C}
\end{equation*}
is given, where $\tPMod{S_g^n}$ is a central extension by $\Z$ of the mapping class group $\PMod{S_g^n}$ of the compact surface of genus $g$ with $n$ boundary components.
The dimension $d$ depends on $g$, $n$, and $\ul$.

These representations are referred to as \emph{quantum representations of mapping class groups}, as their construction goes through
quantum group representation theory. For a fixed pair $(\mathfrak{g},\ell)$, the representations $\rho_g(\ul)$ satisfy strong compatibilities,
and the data of these representations and their compatibilities has been axiomatized as the notion of \emph{modular functor}
(see \cite{bakalovLecturesTensorCategories2000} or \Cref{definitiongeometricmodularfunctor}). Note that there are some modular functors which are not known to arise from a pair $(\Lg,\ell)$,
for example those constructed from the Drinfeld centers of (generalized) Haagerup fusion categories \cite{grossmanDrinfeldCentersFusion2023}.

In this article, we will take the axiomatic approach to modular functors and be blind to how they are constructed.

To a modular functor, one can associate a \emph{modular fusion category}: a fusion tensor category with a braiding and ribbon structure
satisfying a condition called modularity. In fact, the datum of a modular category is equivalent to that of a modular functor
(see \Cref{fullfaithfulnesses} below for a statement and references).

For the modular functor coming from a pair $(\mathfrak{g},\ell)$, the associated modular category is a semisimplification
of the representation category of the quantum group $u_q(\Lg)$ with $q$ a $2m_\Lg\ell$-th root of unity
(see \cite[3.3]{bakalovLecturesTensorCategories2000} for details).

As mapping class groups are fundamental groups of moduli stacks of curves, by the Riemann-Hilbert correspondence,
a quantum representation $\rho_g(\ul)$ can alternatively be seen as a bundle with flat connection over some moduli space of curves.
More precisely, for each modular category $\Cat$, and each choice of $g,n\geq 0$
and objects $x_1,\dotsc,x_n$ of $\Cat$, one gets a bundle with flat connection $(\Nu_g(x_1,\dotsc,x_n),\nabla)$
on some $(\C^*)^{n+1}$-bundle $\tMgb{g}{n}$ over the moduli stack $\Mg{g}{n}$ of genus $g$ curves with $n$ marked points
\cite[6.4]{bakalovLecturesTensorCategories2000}. These bundles with flat connection are usually referred
to as \emph{conformal blocks}.

In the case where the modular functor comes from a pair $(\mathfrak{g},\ell)$,
the associated bundles with flat connection $(\Nu_g(x_1,\dotsc,x_n),\nabla)$ can be constructed directly using the representation
theory of the affine Lie algebra $\widehat{\Lg}$ associated to $\Lg$. The equivalence of this construction
with the quantum group construction mentioned above was proved by Finkelberg \cite{finkelbergEquivalenceFusionCategories1996} and
is based on works of Kazhdan and Lusztig \cite{kazhdanTensorStructuresArising1994}.

In this paper, we will use an alternative and more convenient description of the $(\Nu_g(x_1,\dotsc,x_n),\nabla)$
as bundles with flat connection on some twisted compactifications of moduli spaces of curves. These are proper Deligne-Mumford
stacks. See \Cref{twistedmodulispaces} below for details.

The following table recaps the situation, and also includes the cases of braided fusion categories
and ribbon fusion categories that we did not mention above.

\begin{center}
  \begin{NiceTabular}{ |m{0.12\textwidth}|m{0.2\textwidth}|m{0.3\textwidth}|m{0.25\textwidth}| }
    \toprule
   Category:      & Functor:                  & Representations of:                                             & flat bundles on the compactification: \\[0.25cm] 
   \midrule
   Braided fusion & Braided functor           & Pure braid groups $\PB_n$                                       & $\Mrpo{n}{r}$ (not. \ref{notationonemgr})\\[0.25cm] 
   \midrule
   Ribbon fusion  & Genus $0$ modular functor & Genus $0$ mapping class groups $\PMod{S_0^n}$                   & $\Mgrb{0}{n}{r}$ (rk. \ref{boundary})\\[0.25cm] 
   \midrule
   Modular fusion & Modular functor           & Central $\Z$-extensions $\tPMod{S_g^n}$ of mapping class groups & $\Mgrbt{g}{n}{r}{s}$ (sec. \ref{twistedmodulispaces})\\[0.25cm]  
   \bottomrule
  \end{NiceTabular}
\end{center}


\subsection{Semisimplicity}


Quantum representations constructed from a Lie algebra $\Lg$ and a level $\ell\geq 1$ are known to be semisimple.
Indeed, for each $(\Lg,\ell)$, the monodromies of the associated quantum representations are defined
over a cyclotomic field, and for an appropriate embedding of this field into $\C$,
the representations are unitary and thus semisimple.

Proofs of unitarity either go through the quantum group construction, see
\cite{kirillovInnerProductModular1996,kirillovInnerProductModular1998,wenzlTensorCategoriesQuantum1998}, or through the genus $0$ geometric
construction, see \cite{ramadasHarderNarasimhanTraceUnitarity2009,belkaleUnitarityKZHitchin2012}.
Both of these proofs of unitarity are intricate and involve specific constructions of the modular functor.
This is not surprising, as unitarity of quantum representations does not hold for general modular categories, even up to Galois conjugation.
Consider, for example, the tensor product of the unitary modular category associated to $(\sltwo,3)$ with one of its $2$ non-unitary
Galois conjugates. Consider the associated quantum
representation $\rho_0((2,2),\dotsc,(2,2))$, $n=4$ (in this case $\Lambda=\{0,1,2,3\}^2$).
If it were unitary, its Toledo class, as defined in \cite[4.3]{deroinToledoInvariantsTopological2022}, would vanish.
However, this Toledo class does not vanish, see \cite[5.3]{deroinToledoInvariantsTopological2022}.

There is also a notion of unitarity for modular categories, which implies unitarity of its quantum representations,
see \cite[II.5]{turaevQuantumInvariantsKnots2016}.
Then \cite[Rmk. 8.26]{etingofFusionCategories2005} provides another example of a modular category with no unitary Galois conjugate.

In this article, we give an axiomatic proof of semisimplicity of quantum representations associated to braided, ribbon and modular categories. 

\begin{theorem}\label{mainresult}
  For any modular fusion category, the associated
  representations of central extensions $\tPMod{S_g^n}$ of mapping class groups are semisimple.
  For any ribbon fusion category, the associated representations of genus $0$ mapping class groups $\PMod{S_0^n}$ are semisimple.
  For any braided fusion category, the associated pure braid group representations are semisimple.
\end{theorem}

The question of whether semisimplicity holds in general was asked by Etingof and Varchenko in \cite[4.28]{etingofPeriodicQuasimotivicPencils2023}.
The proof we provide relies on non-Abelian Hodge theory and Ocneanu rigidity. See \Cref{outline} for more details.

\begin{remark}
  Thanks to the recent work \cite{damioliniModularFunctorsConformal2025} of Damiolini and Woike, \Cref{mainresult}
  implies semisimplicity of the projectively flat bundles of conformal blocks associated to any
  strongly rational vertex operator algebra (VOA). Indeed, they show that these
  projectively flat bundles over the moduli space $\widetilde{\mathcal{M}}_{g,n}$ of $n$-marked curves of genus $g$ (with non-zero tangent data) form modular functors.
\end{remark}


\subsection{Application to Hodge structures on conformal blocks}


In the first version of \cite{godfardHodgeStructuresConformal2024}, we proved that for any modular or ribbon fusion category $\Cat$,
the associated bundles with flat connection $(\Nu_g(\ul),\nabla)$ over twisted moduli spaces of curves $\Mgrbt{g}{n}{r}{s}$
support rational variations of Hodge structure over some CM number fields, \emph{provided that they are all semisimple}.
By \Cref{mainresult} above, semisimplicity is always satisfied. Thus, any flat bundle $(\Nu_g(\ul),\nabla)$
associated to a modular or ribbon category supports a rational variation of Hodge structure over a CM number field.
The second version of \cite{godfardHodgeStructuresConformal2024} is updated to reflect this fact%
\footnote{The second version also contains the case of braided fusion categories that was omitted in the first one.}.

Existence of such Hodge structures has the following consequences for quantum representations,
which are precisely the monodromy representations of the $(\Nu_g(\ul),\nabla)$.

\begin{theorem}[{\cite[3.20]{godfardHodgeStructuresConformal2024}}]
  Let $\Cat$ be a modular category and $\rho_g(\ul):\tPMod{S_g^n}\ra\GLn{d}{\C}$ an associated quantum representation.
  Then there exists a CM number field $K$ such that, up to conjugacy, $\rho_g(\ul)$ has image in $\GLn{d}{K}$
  and $\rho_g(\ul)$ preserves a non-degenerate Hermitian form $h$ defined on $K^d$. In particular, there exists $a+b=d$ such that
  $\rho_g(\ul)$ is conjugate in $\GLn{d}{\C}$ to a representation with image in the pseudo-unitary group $U(a,b)\subset \GLn{d}{\C}$.
  The same result holds for representations associated to ribbon or braided fusion categories.
\end{theorem}

It is natural to ask whether these results can be extended to the modular, ribbon or braided category itself. We conjecture the following.

\begin{conjecture}
  Let $\Cat$ be a modular, ribbon or braided fusion category over $\C$. Then $\Cat$ can be defined over some CM number field $K$
  and is Hermitian over $K$.
\end{conjecture}

See \cite[II.5]{turaevQuantumInvariantsKnots2016} for the definition of Hermitian structure on such categories.
The author knows of no example of a braided fusion category which cannot be defined over a cyclotomic field.
In \cite{morrisonNoncyclotomicFusionCategories2012}, examples of fusion categories not definable over cyclotomic fields are constructed,
but they are not braided and their Drinfeld doubles are shown to be definable over cyclotomic fields.


\subsection{Outline of the proof}\label{outline}


We outline the proof for a modular fusion category $\Cat$. The ribbon and braided cases are similar.

The proof is in two steps. The first step uses non-Abelian Hodge theory to construct a continuous family $(\Cat_h)_{h\in\C}$ of modular categories
such that $\Cat_1=\Cat$ and all representations associated to $\Cat_0$ are semisimple.
The second step is to deduce from Ocneanu rigidity that the family $(\Cat_h)_{h\in\C}$ is trivial, i.e. $\Cat_h$ is equivalent to $\Cat$ for all $h$,
thus concluding that all quantum representations associated to $\Cat\simeq\Cat_0$ must be semisimple.

To construct $(\Cat_h)_{h\in\C}$, we use the notion of geometric modular functor, which is equivalent to that of modular fusion category
(\ref{fullfaithfulnesses}).
The geometric modular functor $\Nu$ associated to $\Cat$ consists of bundles with flat connection $(\Nu_g(\ul),\nabla)$
over twisted moduli spaces of curves $\Mgrbt{g}{n}{r}{s}$, \emph{together with some compatibility isomorphisms} (\ref{definitiongeometricmodularfunctor}).

The quantum representations associated to $\Cat$ are then exactly the monodromies of the $(\Nu_g(\ul),\nabla)$.
Hence, to construct the family $(\Cat_h)_{h\in\C}$, we will deform the connections $\nabla$ on the $\Nu_g(\ul)$ to semisimple connections.
To that end, we use the following very general result on the existence of a canonical semisimplification of flat connections
over smooth proper DM stacks. It is a corollary of Simpson's study of non-semisimple local systems on compact Kähler manifolds.

\begin{proposition*}[\ref{semisimplificationstacks}]
  To any flat connection $\nabla$ on a $\Cinf$ bundle $E$ on a smooth proper DM stack $X$ over $\C$
  is associated a canonical family, polynomial in $h\in\C$, of flat connections $\nabla_h=D+\eta_h$ with $D$ a semisimple flat connection, $\nabla_1=\nabla$ and $\nabla_0=D$.
  For each $h\in\C$, the association $(E,\nabla)\mapsto (E,\nabla_h)$ is a functor, compatible with taking tensor products, duals and pullbacks
  along algebraic maps.
\end{proposition*}

The fact that this semisimplification is compatible with tensor products, duals and pullbacks implies that, for each $h\in\C$, the collection
of flat bundles $(\Nu_g(\ul),\nabla_h)$ fits into a geometric modular functor, that we will denote $\Nu_h$.
Note that all monodromy representations of the flat bundles comprising $\Nu_0$ are semisimple.
As the notions of geometric modular functor and modular fusion category are equivalent (see \Cref{fullfaithfulnesses} for details),
we get the desired family $(\Cat_h)_{h\in\C}$ of modular categories.

We can then conclude using Ocneanu rigidity (\cite[2.28]{etingofFusionCategories2005}, \Cref{Ocneanurigidity}).
Indeed, by the Corollary below\footnote{One can see from the definition of $(\Cat_h)_{h\in\C}$ that only the associators vary.},
$\Cat\simeq \Cat_0$ and in particular all quantum representations of $\Cat$ are semisimple.

\begin{corollary*}[Ocneanu rigidity, \ref{corollarydeformations}]
  Let $C$ be a ribbon or braided category over $\C$.
  Then for any continuous family of ribbon or braided fusion categories $(C_t)_{t\in X}$, such that $C_0=C$, with $X$ path-connected,
  and in which only the associators vary, $C_t$ is isomorphic to $C$ for all $t$.
\end{corollary*}


\subsection{Organization of the paper}


In \Cref{sectionfunctors}, we review twisted moduli spaces of curves, and geometric modular/genus $0$ modular/braided functors.
Their relationship with modular/ribbon/braided categories is detailed in \Cref{subsectionfullfaithfulness}.

\Cref{sectionSimpson} reviews what we need from non-Abelian Hodge theory and Simpson's study of extensions of semisimple local systems
on compact Kähler manifolds. The aim of the section is to explain the existence of a canonical semisimplification
of flat connections on smooth proper DM stacks.

\Cref{sectionOcneanu} is devoted to the statement of Ocneanu rigidity, and \Cref{sectionproof} to the proof of \Cref{mainresult}.


\subsection{Acknowledgements}


This paper forms part of the PhD thesis of the author.
The author thanks Julien Marché for his help in writing this paper.
The author is thankful to Yohan Brunebarbe, Pavel Etingof, Nicolas Tholozan, Lukas Woike and Aleksander Zakharov for helpful discussions.


\section{Modular and braided functors}\label{sectionfunctors}


In this section, we review the notions of geometric modular, genus-zero modular and braided functors. A now classical reference is the book
of Bakalov and Kirillov \cite[6.4]{bakalovLecturesTensorCategories2000}. The definitions we use are equivalent to theirs,
but differ in that we use connections on twisted compactifications of moduli spaces of curves instead of regular singular connections.
This simplifies the use of results from non-Abelian Hodge theory.
The section is a rather quick review, and more details can be found in our paper on Hodge structures \cite[Sections 2. and 5.]{godfardHodgeStructuresConformal2024}.


\subsection{Twisted moduli spaces of curves}\label{twistedmodulispaces}


\begin{notation}
  For $g,n\geq 0$ such that $2g-2+n>0$, we will denote by $\Mg{g}{n}$ the moduli stack of smooth curves of genus $g$ with $n$ distinct marked 
  points over $\C$, and by $\oMg{g}{n}$ the Deligne-Mumford compactification classifying stable nodal curves with $n$ distinct marked points
  on their smooth locus.
\end{notation}

We will work with variations of the Deligne-Mumford compactifications.
These depend on an integer $r\geq 1$ and classify $r$-twisted curves in the sense of Kontsevich, Abramovich and Vistoli
\cite{abramovichCompactifyingSpaceStable2002}. For a reference on their definition, see Chiodo's article \cite[1.3]{chiodoStableTwistedCurves2008}.

\begin{notation}
  For $g,n\geq 0$ such that $2g-2+n>0$, denote by $\Mgrp{g}{n}{r}$
  the moduli space of stable nodal $r$-twisted curves of genus $g$ with $n$ distinct smooth marked points \cite[th. 4.4]{chiodoStableTwistedCurves2008}.
\end{notation}

Note that $\Mgrp{g}{n}{1}=\oMg{g}{n}$.

\begin{remark}\label{rootstack}
  One can give an alternative description of these stacks using the root stack construction.
  Let $D_i\subset \oMg{g}{n}$ for $i=1,\dotsc,k$ be the components
  of the boundary divisor. Then $\Mgrp{g}{n}{r}$ is obtained from $\oMg{g}{n}$ by taking $r$-th root stack independently locally
  on each $D_i$ (see \cite[2.3, 4.5]{chiodoStableTwistedCurves2008})
  \begin{equation*}
    \Mgrp{g}{n}{r}=\oMg{g}{n}\left[\sum_i \frac{D_i}{r}\right].
  \end{equation*}
  With this description, by applying the Seifert-van Kampen theorem to an open covering consisting
  of the smooth locus and neighborhoods of boundary divisors, we can see that 
  the fundamental group of $\Mgrp{g}{n}{r}$ is isomorphic
  to the quotient $\PMod{S_{g,n}}/\langle T_\delta^r\mid \delta\rangle$
  of the pure mapping class group of the genus $g$ surface with $n$ \emph{punctured}
  by all $r$-th powers of Dehn twists.
\end{remark}

\begin{notation}\label{notationMgrb}
  For $g,n\geq 0$ such that $2g-2+n>0$, denote by $\Mgrb{g}{n}{r}$
  the moduli space of stable nodal $r$-twisted curves of genus $g$ with $n$ distinct smooth stacky points of order $r$
  \emph{and a section at each stacky point}.
  Here, a stable nodal $r$-twisted curve is a stacky curve
  whose singularities are étale-locally modelled by $\mathrm{Spec}(\C[x,y]/(xy))/\mu_r$ where $\zeta\in\mu_r$ acts
  on $x$ by $\zeta\cdot x=\zeta x$ and $\zeta\cdot y=\zeta^{-1}y$. The curves considered in the definition $\Mgrb{g}{n}{r}$ have
  exactly $n$ stacky points in their smooth loci, numbered from $1$ to $n$ and each of order $r$.
  Moreover, the curves are equipped with $n$ disjoint sections, whose images are located at these $n$ stacky points.
  In particular, each section has an nontrivial automorphism group isomorphic to $\mu_r$.
  We will use the convention $\Mgrb{0}{2}{r}\simeq \Br{r}$.
  See \cite[2.26]{godfardHodgeStructuresConformal2024} for more details.
\end{notation}

\begin{remark}\label{boundary}
  The stack $\Mgrb{g}{n}{r}$ is a $\mu_r^n$-gerbe over $\Mgrp{g}{n}{r}$, which can be described as follows.
  Let $\Sigma_i\subset \Cgrp{g}{n}{r}$ be the $i$-th stacky marked point in the universal curve $\Cgrp{g}{n}{r}\ra \Mgrp{g}{n}{r}$.
  Then $\Sigma_i$ is a $\mu_r$-gerbe over $\Mgrp{g}{n}{r}$ and $\Mgrb{g}{n}{r}$ is the product over $\Mgrp{g}{n}{r}$ of the $\Sigma_i$.
  As explained below, with this description, we find that the fundamental group of $\Mgrb{g}{n}{r}$ is isomorphic
  to the quotient $\PMod{S_g^n}/\langle T_\delta^r\mid \delta\rangle$
  of the pure mapping class group of the
  genus $g$ surface with $n$ \emph{boundary components} by all $r$-th powers of Dehn twists.
  To see this, consider, for $i=1,\dotsc,n$, the line bundle $T_i$ over $\Mgrp{g}{n}{r}$ whose fiber over
  a point is the tangent bundle to the corresponding curve at the $i$-th section, and let $T_i^*$
  denote the total space of $T_i$ minus the zero section. Denote by $\mathcal{M}_g^n$
  the fiber product of the $T_i^*$ over $\mathcal{M}_g^n$. Then the fundamental group of
  $\mathcal{M}_g^n$ is $\PMod{S_g^n}$. Now, for each $i$, $\Sigma_i$
  is the stack of $r$-th roots of $T_i$. It is a general fact that its fundamental group
  is then the quotient of that of $T_i^*$ by the $r$-th power of the monodromy around the zero section.
  Taking fiber products and extending this discussion to the compactification
  then yields the desired identification with
  $\PMod{S_g^n}/\langle T_\delta^r\mid \delta\rangle$.
\end{remark}

The spaces $\Mgrb{0}{n}{r}$ are the ones needed to define genus $0$ modular functors. For full modular functors,
we need yet another variation: the stacks $\Mgrbt{g}{n}{r}{s}$. These spaces also depend on another integer $s\geq 1$, and
$\Mgrbt{g}{n}{r}{s}$ is a $\mu_s$-gerbe over $\Mgrb{g}{n}{r}$ constructed as follows.
Consider the universal curve $\pi:\Cgrb{g}{n}{r}\ra\Mgrb{g}{n}{r}$ and the line bundle
$\mathcal{L}=\det (\pi_*(\omega_\pi))^{\otimes -2}$ over $\Mgrb{g}{n}{r}$,
where $\omega_\pi$ is the relative dualizing sheaf. Then $\Mgrbt{g}{n}{r}{s}$
is the stack of $s$-th roots of $\mathcal{L}$.
At the level of fundamental groups, it corresponds
to the stable central extension of mapping class groups by $\Z$, quotiented by $s\Z$ and by
$r$-th powers of Dehn twists.
This follows by the same argument as in \Cref{boundary}, knowing that the stack $\mathcal{L}^*$
obtained by removing the zero section from the total space of $\mathcal{L}$ has fundamental group
the stable central extension of mapping class groups by $\Z$.
See \cite[2.5]{godfardHodgeStructuresConformal2024} for precise definitions and details on this.

Lastly, we need a slight variation on $\Mgrb{0}{n}{r}$ to define a braided functor.

\begin{notation}\label{notationonemgr}
  For $n\geq 1$, denote by $\Mgrpo{0}{n}{r}$ the moduli space of connected stable nodal $r$-twisted curves
  of genus $0$ with $n$ distinct sections at smooth non-stacky points, and a single smooth stacky point
  of order $r$
  \emph{together with a section at that smooth stacky point}. This is a variation of \Cref{notationMgrb};
  see there for details.
  We will use the convention $\Mgrpo{0}{1}{r}=*$.
\end{notation}

\begin{remark}
  The space $\Mgrpo{0}{n}{r}$ can de identified with $\Sigma_0\subset \Cgrp{0}{n+1}{r}$ as defined in \Cref{boundary},
  where we number the markings form $0$ to $n$. By the same arguments as in
  \Cref{rootstack,boundary}, its fundamental group can be identified with
  $\PMod{D_n}/\langle T_\delta^r\mid \delta\rangle$, where $D_n$ is the closed disk with $n$ points removed.
  The pure mapping class group $\PMod{D_n}$ is the pure braid group $\PB_n$.
\end{remark}

Gluing and forgetful maps between the Deligne-Mumford compactifications $\oMg{g}{n}$ have analogs for the $\Mgrb{g}{n}{r}$
and the $\Mgrbt{g}{n}{r}{s}$. See \Cref{definitiongeometricmodularfunctor} below.
These maps can be defined either using the moduli interpretation
of the stacks or directly from the corresponding maps on
the $\oMg{g}{n}$ using the root stack construction
mentioned in \Cref{rootstack}.


\subsection{Modular functors}


We reproduce here the definition of geometric modular functor of \cite[2.4]{godfardHodgeStructuresConformal2024}.
This definition is equivalent to that given by Bakalov-Kirillov in \cite[6.4.1, 6.7.6]{bakalovLecturesTensorCategories2000}.

\begin{definition}[Modular Functor]\label{definitiongeometricmodularfunctor}
  Let $\Lambda$ be a finite set with involution $\lambda\mapsto\lambda^\dagger$ and preferred fixed point $0\in\Lambda$.
  Let $r,s\geq 1$ be integers.
  Then a geometric modular functor with level $(r,s)$ is the data,
  for each $g,n\geq 0$, $(g,n)\neq (0,0),(0,1),(1,0)$, and $\ul\in\Lambda^n$,
  of a bundle with flat connection $(\Nu_g(\ul),\nabla)$ over $\Mgrbt{g}{n}{r}{s}$, together with some isomorphisms
  described below.
  \begin{description}
      \item[(G-sep)] For each gluing map
      \begin{equation*}
        q:\Mgrbt{g_1}{n_1+1}{r}{s}\times\Mgrbt{0}{2}{r}{s}\times\Mgrbt{g_2}{n_2+1}{r}{s}\lra \Mgrbt{g_1+g_2}{n_1+n_2}{r}{s}
      \end{equation*}
      and each $\ul$, an isomorphism preserving the connections
      \begin{equation*}
        q^*\Nu_{g_1+g_2}(\lambda_1,\dotsc,\lambda_n)\simeq \bigoplus_\mu\Nu_{g_1}(\lambda_1,\dotsc,\lambda_{n_1},\mu)\otimes
        \Nu_0(\mu,\mu^\dagger)^\vee\otimes
        \Nu_{g_2}(\lambda_{n_1+1},\dotsc,\lambda_n,\mu^\dagger);
      \end{equation*}
      \item[(G-nonsep)] For each gluing map
      \begin{equation*}
        p:\Mgrbt{g-1}{n+2}{r}{s}\times\Mgrbt{0}{2}{r}{s}\lra \Mgrbt{g}{n}{r}{s}
      \end{equation*}
      and each $\ul$, an isomorphism preserving the connections
      \begin{equation*}
        p^*\Nu_g(\lambda_1,\dotsc,\lambda_n)\simeq \bigoplus_\mu\Nu_{g-1}(\lambda_1,\dotsc,\lambda_n,\mu,\mu^\dagger)
        \otimes \Nu_0(\mu,\mu^\dagger)^\vee;
      \end{equation*}
      \item[(N)] For each forgetful map $f:\Mgrbt{g}{n+1}{r}{s}\ra\Mgrbt{g}{n}{r}{s}$, and each $\ul$, an isomorphism preserving the connections
      \begin{equation*}
        f^*\Nu_g(\lambda_1,\dotsc,\lambda_n)\simeq \Nu_g(\lambda_1,\dotsc,\lambda_n,0)
      \end{equation*}
      and a canonical isomorphism $(\Nu_0(0,0),\nabla)\simeq (\cO,d)$ (trivial flat bundle);
      \item[(Perm)] For each $\ul\in\Lambda^n$ and permutation $\sigma\in S_n$, an isomorphism preserving the connections
      \begin{equation*}
        \Nu_g(\lambda_1,\dotsc,\lambda_n)\simeq \sigma^*\Nu_g(\lambda_{\sigma(1)},\dotsc,\lambda_{\sigma(n)}).
      \end{equation*}
  \end{description}
  The isomorphisms of \textbf{(G-sep)}, \textbf{(G-nonsep)}, \textbf{(N)} and \textbf{(Perm)}
  are to be compatible with one another and with repeated applications.
  Moreover, we ask for the gluing to be symmetric in the sense that
  for each gluing isomorphism above, the change of variable $\mu\mapsto\mu^\dagger$
  on the right-hand side has the same effect as permuting the summands and applying the \textbf{(Perm)} isomorphisms
  $\Nu_0(\mu,\mu^\dagger)\simeq\Nu_0(\mu^\dagger,\mu)$ induced by $\sigma\in S_2\setminus\{\id\}$.
  
  The functor is also assumed to verify the non-degeneracy axiom
  \begin{description}
      \item[(nonD)] For each $\lambda$, $\Nu_0(\lambda,\lambda^\dagger)\neq 0$.
  \end{description}
\end{definition}

Sometimes, we will shorten $\Nu_0(\ul)$ to $\Nu(\ul)$.
Because $\Mgrbt{g}{n}{r}{s}$ is a $\mu_s$-gerbe, $\mu_s$ acts on the fibers of $\Nu_g(\ul)$.
Using the gluing axiom, one sees that this action is by scalars and independent of $g$ and $\ul$.
Hence to each modular functor $\Nu$ is associated a complex number $c\in\mu_s\subset\C$ corresponding to the action of $e^{2\pi i/s}$,
called the \textbf{central charge} of $\Nu$.

\begin{remark}
  We use somewhat non-standard gluing maps involving $\Mgrbt{0}{2}{r}{s}$: the points are not directly glued together, but are first glued to the marked
  points of a twice-marked sphere, which is then contracted as it is an unstable component in the image. Adding it may seem pointless.
  However, we use it to have a more compact way to deal with the symmetry of the gluing. Indeed, even when $\mu=\mu^\dagger$,
  there are some examples for which the permutation isomorphism $\Nu_0(\mu,\mu)\simeq\Nu_0(\mu,\mu)$ induced by $\sigma\in S_2\setminus\{\id\}$
  is not the identity. See \cite[rmk. 2.7]{godfardHodgeStructuresConformal2024} for more details.
\end{remark}

\begin{definition}[Genus $0$ Modular Functor]
  Let $\Lambda$ be a finite set with involution $\lambda\mapsto\lambda^\dagger$ and preferred fixed point $0\in\Lambda$.
  Let $r\geq 1$ be an integer.
  Then a geometric genus $0$ modular functor with level $r$ is the data,
  for each $n\geq 2$ and $\ul\in\Lambda^n$,
  of a bundle with flat connection $(\Nu(\ul),\nabla)$ over $\Mgrb{0}{n}{r}$ together with isomorphisms
  as in \Cref{definitiongeometricmodularfunctor}, with the $\Mgrbt{0}{n}{r}{s}$ replaced by the $\Mgrb{0}{n}{r}$,
  satisfying all the same axioms\footnote{Note that \textbf{(G-nonsep)} is vacuous in genus $0$.}.
\end{definition}

\begin{remark}\label{remarklevel}
  The choice of level is not significant in the definition of (genus $0$) modular functor.
  Indeed, if $r'$ is a multiple of $r$ and $s'$ a multiple of $s$, pullback along the maps $\Mgrbt{g}{n}{r'}{s'}\ra\Mgrbt{g}{n}{r}{s}$
  produces a modular functor of level $(r',s')$ from a modular functor of level $(r,s)$.
  Similarly for genus $0$ modular functors.
\end{remark}


\subsection{Braided functors}


The gluing and forgetful maps between the moduli spaces $\oMg{0}{n+1}$ induce similar maps between the stacks $\Mrpo{n}{r}$.
See \cite[2.3.3 and Figure 1.1]{godfardConstructionHodgeStructures2024} for more explanations.

\begin{definition}[Braided Functor, compare with {\cite[2.14, 2.15]{godfardHodgeStructuresConformal2024}}]\label{braidedfunctor}
  Let $\Lambda$ be a finite set and $r\geq 1$ an integer.
  A geometric braided functor with level $r$ is the data,
  for each $n\geq 1$, $\ul\in\Lambda^n$ and $\mu\in \Lambda$,
  of a bundle with flat connection $(\Nu(\mu;\ul),\nabla)$ over $\Mrpo{n}{r}$, together with some isomorphisms
  described below.
  \begin{description}
      \item[(G)] For each gluing map
      \begin{equation*}
        q:\Mrpo{n_1+1}{r}\times\Mrpo{n_2}{r}\lra \Mrpo{n}{r}
      \end{equation*}
      and each $\mu,\ul$, an isomorphism preserving the connections
      \begin{equation*}
        q^*\Nu(\mu;\lambda_1,\dotsc,\lambda_n)\simeq \bigoplus_\nu\Nu(\mu;\lambda_{1},\dotsc,\lambda_{n_1},\nu)
        \otimes\Nu(\nu;\lambda_{n_1+1},\dotsc,\lambda_{n});
      \end{equation*}
      \item[(N)] For each forgetful map $f:\Mrpo{n+1}{r}\ra\Mrpo{n}{r}$, and each $\mu,\ul$, an isomorphism preserving the connections
      \begin{equation*}
        f^*\Nu(\mu;\lambda_1,\dotsc,\lambda_n)\simeq \Nu(\mu;\lambda_1,\dotsc,\lambda_n,0),
      \end{equation*}
      and for each $\lambda\in\Lambda$ a canonical isomorphism $(\Nu(\lambda;\lambda),\nabla)\simeq (\cO,d)$ (trivial flat bundle);
      \item[(Dual)] For each $\lambda$ there exists a unique $\mu$ such that $\Nu(0;\lambda,\mu)\neq 0$.
      This $\mu$ will be denoted $\lambda^\dagger$;
      \item[(Perm)] For each $\mu\in\Lambda$, $\ul\in\Lambda^n$ and permutation $\sigma\in S_n$, an isomorphism preserving the connections
      \begin{equation*}
        \Nu(\mu;\lambda_1,\dotsc,\lambda_n)\simeq \sigma^*\Nu(\mu;\lambda_{\sigma(1)},\dotsc,\lambda_{\sigma(n)}).
      \end{equation*}
  \end{description}
  The isomorphisms of \textbf{(G)}, \textbf{(N)} and \textbf{(Perm)}
  are to be compatible with each other and repeated applications.
\end{definition}

For the isomorphisms of \textbf{(N)}, we remind the reader of the convention $\Mrpo{1}{r}=*$.
\Cref{remarklevel} on the unimportance of the choice of the level also applies to geometric braided functors.


\subsection{Relation with modular/ribbon/braided fusion categories}\label{subsectionfullfaithfulness}


For definitions of monoidal category, rigid monoidal category, fusion category, braided monoidal category,
twist on a braided monoidal category and modular category, see \cite[2.1, 2.10, 4.1, 8.1, 8.10 and 8.13]{etingofTensorCategories2015}.
In these notes, the base field for such categories will always be $\C$.

The natural category of braided tensor categories is not a $1$-category. However, here we work with fusion categories and
we only care about isomorphisms of fusion categories with structure, not all morphisms. In this context,
we can produce a groupoid of braided/ribbon/modular fusion categories that is a $1$-groupoid by imposing restrictions on
how morphisms behave on the underlying structures of linear category.
These restrictions are benign: any braided/ribbon/modular fusion category is equivalent to one in the groupoids $\Braid/\Rib/\Modcat$ defined below.

\begin{notation}
  For $\Lambda$ a finite set, $\VL$ denotes the category $(\mathrm{Vect}_\C)^\Lambda$ of $\Lambda$-colored vector spaces.
  The object of $\VL$ corresponding to the vector space $\C$ colored by $\lambda\in\Lambda$ will be denoted $[\lambda]$.
\end{notation}

\begin{definition}[{\cite[5.5]{godfardHodgeStructuresConformal2024}}]\label{definitionribbon}
  Let $\Rib$ be the groupoid whose:
  \begin{description}
    \item[(1)] objects are pairs $(\Lambda,\VL)$, where $\Lambda$ is some finite set, and $\VL$ is endowed with
    the structure of a ribbon fusion category;
    \item[(2)] morphisms from $(\Lambda_1,\VLone)$ to $(\Lambda_2,\VLtwo)$ are pairs $(f,\phi)$, with $f:\Lambda_1\ra\Lambda_2$ a bijection and 
    $\phi:\otimes_1\simeq f^*\otimes_2$ a natural isomorphism, such that they induce a monoidal isomorphism $\VLone\ra\VLtwo$
    compatible with braidings and twists. In other words, morphisms form $(\Lambda_1,\VLone)$ to $(\Lambda_2,\VLtwo)$ are ribbon isomorphisms that
    are induced by a bijection $f:\Lambda_1\ra\Lambda_2$ at the level of the $\C$-linear categories and are strict on tensor units.
  \end{description}
  One similarly defines the groupoid $\Braid$ of braided fusion categories and the full sub-groupoid $\Modcat\subset \Rib$
  of modular fusion categories.
\end{definition}

\begin{remark}
  A fusion structure on $\VL$ determines a unique involution $\lambda\mapsto\lambda^\dagger$ of $\Lambda$
  and a preferred fixed point $0\in\Lambda$ of this involution, by imposing that for each $\lambda$, $[\lambda]^*\simeq [\lambda^\dagger]$
  and $1\simeq [0]$.
\end{remark}

\begin{definition}
  For $c\in\C^\times$, we will denote by $\GMod_c$ the groupoid of geometric modular functors with central charge $c$,
  where we identify modular functors of different levels as in \Cref{remarklevel}.

  An isomorphism between $\Nu$ and $\Nu'$ in $\GMod_c$ is a bijection $\phi:\Lambda\simeq\Lambda'$ preserving the involution and $0$ together with
  a family of isomorphisms $\Nu_g(\ul)\simeq\Nu'_g(\phi(\ul))$ compatible with gluing, forgetful, normalization and permutation isomorphisms.

  We will use the notation $\GMod^0$ for the groupoid of genus $0$ geometric modular functors, up to change of level,
  and the notation $\BraidFun$ for the groupoid of geometric braided functors, up to change of level.
\end{definition}

One can associate to a geometric (genus $0$) modular functor over $\C$ a ribbon weakly fusion category and to a geometric braided functor over $\C$ a
braided weakly fusion category, see \cite[5.5.1]{bakalovLecturesTensorCategories2000} and \cite[5.]{godfardHodgeStructuresConformal2024}.
Moreover, Etingof and Penneys recently proved the long-standing conjecture that braided weakly fusion categories are braided fusion categories
\cite{etingofRigidityNonnegligibleObjects2024}. Together with the work of Bakalov and Kirillov \cite[5.4.1, 6.7.13]{bakalovLecturesTensorCategories2000},
this implies that one can associate a modular fusion category to every modular functor, a ribbon fusion category to every genus $0$ modular functor,
and a braided fusion category to every braided functor.
These can be put into the commutative diagram of categories below.
For more details, see \cite[5.]{godfardHodgeStructuresConformal2024}. 
\[\begin{tikzcd}
	{\bigsqcup_c\GMod_c} & {\GMod^0} & \BraidFun \\
	\Modcat & \Rib & {\Braid.}
	\arrow[from=1-1, to=1-2]
	\arrow["{\sqcup_c\mathrm{f}_c}"', from=1-1, to=2-1]
	\arrow[from=1-2, to=1-3]
	\arrow["{\mathrm{f}^0}", from=1-2, to=2-2]
	\arrow["{\mathrm{f}^b}", from=1-3, to=2-3]
	\arrow[from=2-1, to=2-2]
	\arrow[from=2-2, to=2-3]
\end{tikzcd}\]

The Reshetikhin-Turaev construction \cite{reshetikhinInvariants3manifoldsLink1991} provides a functor in the opposite direction:
given a modular category, it provides a $2+1$ dimensional TQFT and hence a modular functor. Here is a statement encapsulating what we will need from this construction
and its simpler genus $0$ variants.

\begin{theorem}[{\cite[5.9]{godfardHodgeStructuresConformal2024}, based on \cite[5.4.1, 6.7.13]{bakalovLecturesTensorCategories2000}}]\label{fullfaithfulnesses}
  The functors $\mathrm{f}^0$ and $\mathrm{f}^b$ are equivalences.
  The functors $\mathrm{f}_c$ are fully faithful, and every object of $\Modcat$
  is in the essential image of some $\mathrm{f}_c$ for a suitable value\footnote{See \cite[5.7.10]{bakalovLecturesTensorCategories2000} for the exact values $c$ may take.}
  of $c$.
\end{theorem}

To end this section, let us say a word about how the functor $\mathrm{f}^b$ and its quasi-inverse are defined. A more detailed account is given in
\cite[5.]{godfardHodgeStructuresConformal2024}. Let $\Nu$ be a braided functor on the finite set $\Lambda$.
Then we can define a tensor structure on $\VL$ as follows. The functor $\otimes:\VL^{\boxtimes 2}\ra\VL$ is defined by
\begin{equation*}
  [\lambda]\otimes[\mu] = \bigoplus_\nu \Nu(\nu;\lambda,\mu)[\nu].
\end{equation*}
To define the associator, note that by parallel transport along the connection on $\Nu(\mu;\lambda_1,\lambda_2,\lambda_3)$,
we can identify the stalk of this bundle at two of the boundary points $x$ and $y$ by the path in \Cref{path}. We then define
the associator $\Hom{\mu}{(\lambda_1\otimes\lambda_2)\otimes\lambda_3}\simeq\Hom{\mu}{\lambda_1\otimes(\lambda_2\otimes\lambda_3)}$
as
\begin{align*}
  \Hom{\mu}{(\lambda_1\otimes\lambda_2)\otimes\lambda_3} &= \bigoplus_\delta \Nu(\mu;\delta,\lambda_3)\otimes\Nu(\delta;\lambda_1,\lambda_2) \\
                                                         &\simeq \Nu(\mu;\lambda_1,\lambda_2,\lambda_3)_x \\
                                                         &\simeq \Nu(\mu;\lambda_1,\lambda_2,\lambda_3)_y \\
                                                         &\simeq \bigoplus_\delta \Nu(\mu;\lambda_1,\delta)\otimes\Nu(\delta;\lambda_2,\lambda_3) \\
                                                         &= \Hom{\mu}{\lambda_1\otimes(\lambda_2\otimes\lambda_3)}.
\end{align*}

\begin{figure}
  \def\svgwidth{0.4\linewidth}
  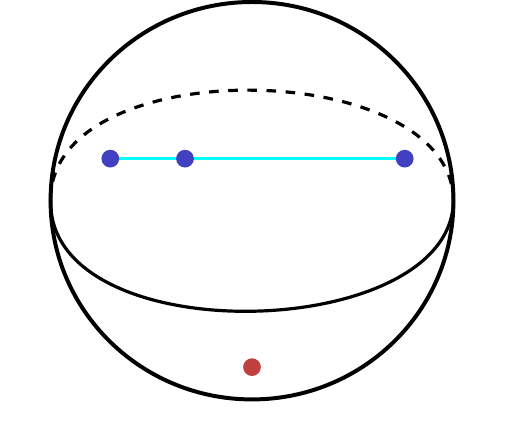
  \caption{Path in $\Mgrpo{0}{3}{r}$ corresponding to moving the point marked $\lambda_2$ from the point marked $\lambda_1$ to the point marked $\lambda_3$.}
  \label[figure]{path}
\end{figure}

The braiding isomorphisms
\begin{equation*}
  \Hom{\mu}{\lambda_1\otimes\lambda_2}=\Nu(\mu;\lambda_1,\lambda_2)\simeq \Nu(\mu;\lambda_2,\lambda_1)=\Hom{\mu}{\lambda_2\otimes\lambda_1}
\end{equation*}
are then essentially given by the path braiding the two marked points in $\Mrpo{2}{r}$
(see \cite[5.]{godfardHodgeStructuresConformal2024}).
To be more precise, as we use moduli of twisted curves, we can make sense of this as follows.
The stack $\Mrpo{2}{r}$ is isomorphic to $\Br{r}$ and has a natural $S_2=\Z/2\Z$ action given
by changing the order of the markings. The quotient can naturally be identified with $\Br{2r}$.
The bundle $\bigoplus_{\lambda_1,\lambda_2}\Nu(\mu;\lambda_1,\lambda_2)$ over $\Mrpo{2}{r}$
has a natural $S_2=\Z/2\Z$ action compatible with that on $\Mrpo{2}{r}$, given by the \textbf{(Perm)} axiom in \Cref{braidedfunctor}.
Hence $\bigoplus_{\lambda_1,\lambda_2}\Nu(\mu;\lambda_1,\lambda_2)$ descends to $\Mrpo{2}{r}/S_2=\Br{2r}$.
The action of the braiding is then given by the action of $e^{i\pi/r}\in\mu_{2r}$ on the fibers of
$\bigoplus_{\lambda_1,\lambda_2}\Nu(\mu;\lambda_1,\lambda_2)$.

We now say a few words about a quasi-inverse to $f^b$. Let $\Cat$ be a braided fusion category with underlying $\C$-linear category $\VL$.
For each choice of $\mu,\lambda_1,\dotsc,\lambda_n\in\Lambda$, the braiding and the associator of $\Cat$ yield
an action of the pure braid group $\PB_n$ on $\hom_\Cat([\mu],[\lambda_1]\otimes\dotsb\otimes[\lambda_n])$.
One can show (see \cite[2.17, 2.18]{godfardHodgeStructuresConformal2024})
that this action factors through $\PB_n/\langle T_\delta^r\mid \delta\rangle=\pi_1\left(\Mrpo{n}{r}\right)$
for some $r$ depending only on $\Cat$.
The Riemann-Hilbert correspondence provides a bundle with flat connection $\Nu(\mu;\lambda_1,\dotsc,\lambda_n)$ whose monodromy
is given by the above pure braid group action on $\hom_\Cat([\mu],[\lambda_1]\otimes\dotsb\otimes[\lambda_n])$.
The gluing isomorphisms are then given by partial compositions in $\Cat$, as mentioned in \Cref{introquantumrep},
while forgetful isomorphisms are induced by the tensor unit $[0]$ of $\Cat$.


\section{Semisimplification of local systems}\label{sectionSimpson}


The purpose of this section is to recap Simpson's study of local systems on compact Kähler manifolds as
formal extensions of semisimple local systems \cite[§3]{simpsonHiggsBundlesLocal1992},
and to deduce from it a canonical semisimplification of
local systems on smooth proper DM stacks. Below is an overview.

Any flat connection $\nabla$ on a smooth bundle $E$ can be described as an iterated extension of semisimple flat bundles.
From this one can get a decomposition $\nabla = D+\eta$ where $D$ is a \emph{semisimple} flat connection
and $\eta$ is a $1$-form with values in $\End{E}$ describing the iterated extension (we call such $\eta$ an extension).
However, this decomposition is not unique. The point of this section is to describe a canonical choice of decomposition
provided by non-Abelian Hodge theory when the base is compact Kähler.
This is essentially part of Simpson's paper. In short, the canonical decomposition $\nabla = D+\eta$ is characterized
by the fact that $\eta$ is killed by an operator $D'$ associated to $D$.
This canonical decomposition turns out to be functorial and compatible with natural operations
on flat bundles. We call this choice of $D$ \emph{the} semisimplification of $\nabla$.

We also describe a canonical deformation to the semisimplification, in the form of a polynomial family $(\eta_h)_{h\in\C}$
such that $\eta_1=\eta$, $\eta_0=0$ and $D+\eta_h$ is flat for all $h$. This family is provided by a
bijection between the extensions $\eta$ and extension cohomology classes, which gives the structure of a quadratic $\C$-cone on the set of extensions $\eta$.

Finally, using Simpson's formalism \cite{simpsonLocalSystemsProper2011}, we deduce the same results for smooth proper algebraic DM stacks.

The references for this section are \cite[§3]{simpsonHiggsBundlesLocal1992} for compact Kähler manifolds,
and \cite{simpsonLocalSystemsProper2011} for stacks.


\subsection{Formality of the dg-category of semisimple local systems}


In \cite{deligneRealHomotopyTheory1975}, Deligne, Griffiths, Morgan and Sullivan describe a two-term zigzag proving formality as a commutative dg-algebra
of the real de Rham complex of a compact Kähler manifold.
This zigzag involves looking at the kernel and the cohomology of an operator $d^c$ on the de Rham complex which commutes with the differential $d$
and satisfies a $dd^c$-Lemma (principle of two types).
Over $\C$, one may replace $d^c$ by the holomorphic differential operator $\partial$ in their argument.
In \cite[§3]{simpsonHiggsBundlesLocal1992}, Simpson extends this method to prove a formality result for de Rham complexes with coefficients
in a \emph{semisimple} complex local system. More precisely, the formality is that of the dg-category of semisimple flat bundles.
This generalizes \cite{deligneRealHomotopyTheory1975} since the de Rham complex of the manifold is the endomorphism complex of the rank $1$ trivial bundle.
The operator $\partial$ on the complex of the trivial bundle generalizes to an operator $D'$ on the complex of any semisimple local system.
In the rest of this subsection, we recap the proof of this result.

Let $X$ be a compact Kähler manifold and $(E,D)$ a complex $\Cinf$-vector bundle with a flat connection whose monodromy is semisimple.
Then non-Abelian Hodge theory provides a unique decomposition $D=D'+D''$ of the connection into $2$ operators \cite[§1]{simpsonHiggsBundlesLocal1992}.
Here $D'$ is a flat $\partial$-connection
and $D''$ is a flat $\overline{\partial}$-connection, i.e. $D'(fs)=s\otimes \partial f+ fD's$ and $D''(fs)=s\otimes \overline{\partial} f+ fD''s$
for any section $s$ and function $f$. This decomposition is constructed through the choice of a harmonic metric on $(E,D)$,
see \cite[§1]{simpsonHiggsBundlesLocal1992} for details. Moreover, this decomposition is compatible with maps of semisimple flat bundles,
with taking tensor products and duals, and with pullbacks along complex maps $Y\ra X$ between compact Kähler manifolds.

The operators $D$, $D'$ and $D''$ are flat and thus each extends to a differential on the smooth de-Rham complex $\cA^\bullet(E)$ of $E$.
In fact, they satisfy Kähler identities \cite[§2]{simpsonHiggsBundlesLocal1992}, from which the following $DD'$-Lemma is deduced.

\begin{lemma}[$DD'$-Lemma {\cite[2.1]{simpsonHiggsBundlesLocal1992}}]\label{ddp}
  Let $\alpha\in\cA^k(E)$ be in the kernel of both $D$ and $D'$.
  Then if $\alpha$ is in the image of $D$ or $D'$, there exists $\beta\in \cA^{k-2}(E)$
  such that $DD'\beta=\alpha$.
\end{lemma}

Note that as $D''=D-D'$ is flat, $DD'=\pm D'D$ on $\cA^k(E)$.

\begin{example}
  The simplest case is when the connection $D$ preserves a unitary metric on $E$.
  Then, $D'$ and $D''$ are just respectively the $(1,0)$ and $(0,1)$ parts of
  $D:\cA^0(E)\ra \cA^1(E)=\cA^{1,0}(E)\oplus \cA^{0,1}(E)$.
  For example, if $(E,D)$ is the rank $1$ trivial bundle, $D=d$, $D'=\partial$ and $D''=\dol$.
  Then the $DD'$-Lemma above is essentially the usual $\partial\dol$-Lemma in Hodge theory.
\end{example}

Let us now explain the main ingredient in Simpson's study of (iterated) extensions of semisimple local systems.

\begin{lemma}[Formality {\cite[2.2]{simpsonHiggsBundlesLocal1992}}, see also {\cite{deligneRealHomotopyTheory1975}}]\label{zigzag}
  Denote by $\ZD^\bullet(E)\subset \cA^\bullet(E)$ the subset of $D'$-closed forms
  and by $\HD^\bullet(E)$ the cohomology of $(\cA^\bullet(E),D')$.
  Then the arrows in the diagram
  $$(\cA^\bullet(E),D)\longleftarrow (\ZD^\bullet(E),D)\lra (\HD^\bullet(E),0)$$
  are quasi-isomorphisms of complexes. These are compatible with taking tensor products, duals and pullbacks of flat bundles
  and with maps of flat bundles.
\end{lemma}
\begin{proof}
  From the fact that $D$ and $D'$ commute, one sees that the arrows are well defined maps of complexes,
  provided that the differential induced by $D$ on $\HD^\bullet(E)$ is $0$, which we now check.
  Let $\alpha$ be a $D'$-closed $k$-form. We want to show that $D\alpha$ is $D'$-exact.
  To this end note that $D\alpha$ satisfies the $DD'$-Lemma and thus that $D\alpha=D'D\beta$ for some $(k-1)$-form $\beta$.
  Hence $D\alpha$ is $D'$-exact, and the induced differential is indeed $0$.

  Let us now show that the arrow on the left is a quasi-isomorphism.
  Let $\alpha$ be a $D$-closed $k$-form. Then $D'\alpha$ satisfies the $DD'$-Lemma,
  and thus $D'\alpha=D'D\beta$ for some $(k-1)$-form $\beta$. Then $\alpha-D\beta$ is $D'$-closed
  and has the same $D$-cohomology class as $\alpha$. This shows that the map is surjective on cohomology.
  For injectivity, assume that $D\beta$ is a $D'$-closed $k$-form. We want to show that $D\beta=D\gamma$
  for a $D'$-closed $\gamma$. Applying the $DD'$-Lemma to $D\beta$ yields $D\beta=DD'\delta$ for some $\delta$.
  Setting $\gamma=D'\delta$ concludes.

  Let us turn our attention to the arrow on the right.
  Let $\alpha$ be a $D'$-closed $k$-form. Applying the $DD'$-Lemma, $D\alpha=DD'\beta$ for some $\beta$.
  Then $\alpha-D'\beta$ is $D$-closed and has the same image as $\alpha$ in $\HD^k(E)$. This shows surjectivity on cohomologies.
  For injectivity, assume that $D'\alpha$ is a $D$-closed $k$-form. We want to show that $D'\alpha$ is also $D$-exact.
  But this is just the $DD'$-Lemma again.

  The compatibility of the arrows with respect to the operations mentioned is directly deduced from that of the decomposition
  $D=D'+D''$.
\end{proof}

\begin{remark}
  In particular, note that $(\cA^\bullet(E),D)$ and $(\cA^\bullet(E),D')$ have canonically identified cohomologies.
\end{remark}

We use Simpson's conventions for dg-categories \cite[§3]{simpsonHiggsBundlesLocal1992}:
a dg-category $\Cat$ is a category enriched in cochain complexes of $\C$-vector spaces concentrated in nonnegative degrees.
For the notions of functors, natural transformations and quasi-equivalences in the context of dg-categories, see there.

\begin{notation}\label{notationzero}
  For $\Cat$ a dg-category, we will denote by $\Cat^0$ the associated $\C$-linear category.
  $\Cat^0$ has the same objects as $\Cat$ with morphisms $\Cat^0(x,y)=\mathrm{H}^0(\Cat(x,y))$.
\end{notation}

For each compact Kähler manifold $X$, the following dg-categories will be relevant to this paper
\cite[3.4.1-3.4.4]{simpsonHiggsBundlesLocal1992}.
\begin{itemize}
  \item The dg-category $\CdR$ whose objects are flat bundles on $X$ and
whose complex of morphisms from $(E,\nabla^E)$ to $(F,\nabla^F)$ is $\cA^\bullet(\Hom{E}{F})$
with the differential induced by the connection on $\Hom{E}{F}$. We will denote by $\CdRs$ the full sub-dg-category
of semisimple flat bundles.

  \item The dg-category $\CD$ whose objects are \emph{semisimple} flat bundles on $X$
and whose complex of morphisms from $(E,D^E)$ to $(F,D^F)$ is $\ZD^\bullet(\Hom{E}{F})$
with the differential $D$ induced by the connection on $\Hom{E}{F}$.

  \item The dg-category $\CH$ whose objects are \emph{semisimple} flat bundles on $X$
and whose complex of morphisms from $(E,D^E)$ to $(F,D^F)$ is $\HD^\bullet(\Hom{E}{F})$
with the $0$ differential.
\end{itemize}

As a direct corollary of \Cref{zigzag}, we have the following functorial formality result.

\begin{proposition}[{\cite[3.4 and p.43]{simpsonHiggsBundlesLocal1992}}]\label{quasiequiv}
  The maps of \Cref{zigzag} induce quasi-equivalences of dg-categories
  $$\CdRs\longleftarrow \CD \lra \CH.$$
  Moreover these quasi-equivalences are compatible with the natural symmetric tensor and rigid structures induced by
  taking tensor products and duals of flat bundles, and with respect to pullbacks along complex maps $Y\ra X$ between compact Kähler manifolds.
\end{proposition}


\subsection{The dg-formalism for extensions}


\begin{definition}[Extension {\cite[§3]{simpsonHiggsBundlesLocal1992}}]
  An extension in a dg-category $\Cat$ is a diagram
  $$x\xrightarrow{a} y\xrightarrow{b} z$$
  where $a$ and $b$ are closed degree $0$ maps, $ba=0$ and there exist degree $0$ maps $f:y\ra x$ and $g: z\ra y$
  such that $fg=0$, $fa=\id_x$, $bg=\id_z$ and $af+gb=\id_y$.

  The associated extension class is defined to be $[fd(g)]\in \mathrm{H}^1(\Cat(z,x))$. Here $d$ denotes the differential of $\Cat(z,y)$.
\end{definition}

With the notation of the definition, we will say that $y$ is an extension of $z$ by $x$.
The extension class is independent of the choice of $(f,g)$, and two extension diagrams of $z$ by $x$
admit an isomorphism (relative to $z$ and $x$) if and only if they have
the same extension class.

Note that in $\CdR$, the notions of extensions and extension classes coincide with the usual notions
of extensions and extension classes for flat bundles.

\begin{definition}[Extension completion {\cite[§3]{simpsonHiggsBundlesLocal1992}}]
  Let $\Cat$ be a dg-category, whose differentials we will denote by $d$.
  Let $\overline{\Cat}$ be the dg-category whose
  \begin{itemize}
    \item objects are pairs $(x,\eta)$ where $x$ is an object of $\Cat$ and $\eta\in \Cat^1(x,x)$
    is a Maurer-Cartan element, i.e. $d\eta +\eta^2=0$;
    \item complex of morphisms form $(x,\eta)$ to $(y,\xi)$ is $\Cat(x,y)$
    with the twisted differential $\hat{d}f=df+\xi f- (-1)^{\mathrm{deg} f} f\eta$.
  \end{itemize}
  Then the completion for extension $\hat{\Cat}$ is the full sub-dg-category of $\overline{\Cat}$
  consisting of objects which are iterated extensions of objects of the form $(x,0)$.
\end{definition}

Note that there is a natural embedding $\Cat\ra \hat{\Cat}$ induced by $x\mapsto (x,0)$,
and that by definition every object of $\hat{\Cat}$ is an iterated extension of objects in $\Cat$.
The following lemma explains why $\hat{\Cat}$ is called the completion for extensions.

\begin{lemma}[{\cite[3.1]{simpsonHiggsBundlesLocal1992}}]\label{complete}
  Let $\Cat$ be a dg-category. If for every pair of objects $x$, $z$, every class $\omega\in \mathrm{H}^1(\Cat(z,x))$
  is the class of some extension $x\ra y\ra z$, then the embedding $\Cat\ra \hat{\Cat}$ is a quasi-equivalence.
\end{lemma}

A dg-category $\Cat$ satisfying the hypothesis of \Cref{complete} is called \emph{extension-complete}.

\begin{example}\label{completionCdRs}
  Fix $X$. The dg-category $\CdR$ is extension-complete. Indeed, if
  $$\omega\in \mathrm{H}^1\left(\CdR\left((F,\nabla^F),(E,\nabla^E)\right)\right)$$
  is a class, then any representative $\eta$ of $\omega$ is a $\nabla$-closed $1$-form in $\cA^1(\Hom{F}{E})$ and thus induces an extension
  $$\left(E\oplus F,\begin{pmatrix} \nabla^E & \eta \\ 0 & \nabla^F \end{pmatrix}\right)$$
  of $(F,\nabla^F)$ by $(E,\nabla^E)$.

  However, its sub-dg-category $\CdRs$ of semisimple flat bundles is not complete. In fact, its completion for extensions
  $\hCdRs$ is quasi-equivalent to $\CdR$. Indeed, one can check directly that the functor
  $$\hCdRs\lra \CdR,\: \left((E,D),\eta\right)\mapsto (E,D+\eta)$$
  is a quasi-equivalence.
\end{example}

The last Lemma we will need from Simpson's paper is the following.

\begin{lemma}[{\cite[3.3]{simpsonHiggsBundlesLocal1992}}]
  A quasi-equivalence $F:\Cat_1\ra \Cat_2$ of dg-categories
  induces a quasi-equivalence of completions for extensions
  $$\hat{F}:\hat{\Cat}_1\ra \hat{\Cat}_2,\: (x,\eta)\mapsto (Fx,F\eta).$$
\end{lemma}

Applying this Lemma to \Cref{quasiequiv}, we get the following.

\begin{theorem}[{\cite[3.4]{simpsonHiggsBundlesLocal1992}}]\label{equivalences}
  Let $X$ be a compact Kähler manifold. Then there are quasi-equivalences of dg-categories
  \[\begin{array}{ccccc}
    \CdR        & \longleftarrow & \hCD                    & \lra        & \hCH \\
    (E,D+\eta)  & \longmapsfrom  & \left((E,D),\eta\right) & \longmapsto & \left((E,D),[\eta]\right).
  \end{array}\]
  Here $[\eta]$ denotes the $D'$-cohomology class of $\eta$. Moreover, as in \Cref{quasiequiv},
  these equivalences are compatible with the symmetric tensor and rigid structures of the dg-categories, and with respect
  to pullbacks along complex maps $Y\ra X$ between compact Kähler manifold.
\end{theorem}


\subsection{Canonical semisimplification of local systems}


In this subsection, we explain how Simpson's Theorem (\ref{equivalences}) provides a canonical semisimplification
of flat bundles on compact Kähler manifolds, and a canonical deformation to the semisimplification.
We also show that this deformation is polynomial and hence continuous. We fix a compact Kähler manifold $X$.
All constructions below will be functorial with respect to complex maps $Y\ra X$ between compact Kähler manifolds.

Let us first give an explicit description of morphisms in the categories associated to $\hCD$ and $\hCH$
by taking $\mathrm{H}^0$ on morphism complexes (see \Cref{notationzero}).

\begin{proposition}\label{morphisms}
  The category $\hCDz$ associated to $\hCD$ is the category whose
  \begin{itemize}
    \item objects are pairs $\left((E,D),\eta\right)$ with $(E,D)$ a semisimple flat bundle and $\eta\in\cA^1(\End{E})$,
    such that $\eta$ is an iterated extension, $D\eta+\eta^2=0$ and $D'\eta=0$;
    \item morphisms from $\left((E,D^E),\eta\right)$ to $\left((F,D^F),\xi\right)$ are morphisms of flat bundles
    $f:(E,D^E)\ra (F,D^F)$ such that $f\eta=\xi f$.
  \end{itemize}
  Similarly, the category $\hCHz$ associated to $\hCH$ is the category whose
  \begin{itemize}
    \item objects are pairs $\left((E,D),u\right)$ with $(E,D)$ a semisimple flat bundle and $u\in\HD^1(\End{E})$,
    such that $u$ is an iterated extension and $u^2=0$ in $\HD^2(\End{E})$;
    \item morphisms from $\left((E,D^E),u\right)$ to $\left((F,D^F),v\right)$ are morphisms of flat bundles
    $f:(E,D^E)\ra (F,D^F)$ such that $fu=vf$ in $\HD^1(\Hom{E}{F})$.
  \end{itemize}
\end{proposition}
\begin{proof}
  The description of the objects follow from the definitions.
  Morphisms from $\left((E,D^E),\eta\right)$ to $\left((F,D^F),\xi\right)$ in $\hCDz$
  are elements of $\mathrm{H}^0(\ZD^\bullet(\Hom{E}{F}),\hat{d})$, where the differential
  $\hat{d}$ on function is given by $\hat{d}f=Df+\xi f-f\eta$.
  Now by \Cref{zigzag}
  $$\mathrm{H}^0(\cA^\bullet(\Hom{E}{F}),D)=\mathrm{H}^0(\cA^\bullet(\Hom{E}{F}),D').$$
  Hence any $D'$-closed $f$ in $\cA^0(\Hom{E}{F})$ is automatically $D$-closed,
  and thus a morphism of the underlying flat bundles.
  This in particular applies to elements of $\mathrm{H}^0(\ZD^\bullet(\Hom{E}{F}),\hat{d})$,
  and explains the above description of morphisms in $\hCD$.
  The same proof applies to morphisms in $\hCH$.
\end{proof}

From \Cref{equivalences} and \Cref{morphisms}, we deduce the following isomorphisms of categories.

\begin{corollary}
  Let $X$ be a compact Kähler manifold. Then there are \emph{isomorphisms of categories}
  \[\begin{array}{ccccc}
    \CdRz        & \longleftarrow & \hCDz                    & \lra        & \hCHz \\
    (E,D+\eta)   & \longmapsfrom  & \left((E,D),\eta\right)  & \longmapsto & \left((E,D),[\eta]\right).
  \end{array}\]
  Here $[\eta]$ denotes the $D'$-cohomology class of $\eta$. As in \Cref{equivalences},
  these equivalences are compatible with the symmetric tensor and rigid structures of the dg-categories, and with respect
  to pullbacks along complex maps $Y\ra X$ between compact Kähler manifolds.
\end{corollary}

\begin{proof}
  Quasi-equivalences of dg-categories induce equivalences of categories when taking $\mathrm{H}^0$ on morphism complexes.
  Hence both functors in the diagram are equivalences. It remains to check that they induce bijections on objects.
  Let $(E,\nabla)$ be an object in $\CdRz$. It is isomorphic to some $(F,D+\eta)$ for $\left((F,D),\eta\right)$
  an object of $\hCDz$. Let $f:(E,\nabla)\ra (F,D+\eta)$ be an isomorphism.
  Then $\left((E,f^*D),f^*\eta\right)$ is an object of $\hCDz$ mapping to $(E,\nabla)$. This shows surjectivity of the first functor.
  Let $\left((E,D^E),\eta\right)$ and $\left((F,D^F),\xi\right)$ be two objects of $\hCDz$ mapping to the same object in
  $\CdRz$. Then $E=F$ and $D^E+\eta=D^F+\xi$. In particular, $\id_E:(E,D^E+\eta)\ra (E,D^F+\xi)$ is an isomorphism in $\CdRz$.
  So, by full faithfulness, $\id_E:\left((E,D^E),\eta\right)\ra \left((E,D^F),\xi\right)$ is an isomorphism in $\hCDz$.
  By the description of morphisms in $\hCDz$ provided by \Cref{morphisms}, $D^E=D^F$ and $\eta=\xi$. This proves
  injectivity of the first functor.

  Let $\left((E,D),\omega\right)$ be an object in $\hCHz$. There exists an object $\left((F,D^F),\eta\right)$ in $\hCDz$
  and an isomorphism $f:(E,D)\ra (F,D^F)$ such that $f^*[\eta]=\omega$. Then $\left((E,D),f^*\eta\right)$
  is an object of $\hCDz$ mapping to $\left((E,D),\omega\right)$ in $\hCHz$. This shows surjectivity of the second functor.
  As for injectivity, assume that $\left((E,D^E),\eta\right)$ and $\left((F,D^F),\xi\right)$ are two objects
  of $\hCDz$ mapping to the same object in $\hCHz$. Then by the description of morphisms in
  $\hCHz$ provided by \Cref{morphisms}, $E=F$ and $D^E=D^F$. As above, $\id_E:\left((E,D^E),[\eta]\right)\ra \left((E,D^E),[\xi]\right)$
  is then an isomorphism, as $[\eta]=[\xi]$, which must lift to an isomorphism in $\hCDz$ by full faithfulness.
  Hence $\eta=\xi$. This shows injectivity of the second functor.
\end{proof}

From these isomorphisms of categories and the description of morphisms in $\hCDz$,
we immediately get a canonical semisimplification for bundles with flat connections
on compact Kähler manifolds.

\begin{corollary}\label{semisimplification}
  Any flat connection $\nabla$ on a $\Cinf$ bundle $E$ on a compact Kähler manifold $X$
  has a unique decomposition $\nabla=D+\eta$ with $D$ a flat semisimple connection on $E$
  and $\eta\in\cA^1(\End{E})$ such that $\eta$ is an iterated extension and $D'\eta=0$.
  This decomposition is compatible with taking tensor products, duals and pullbacks,
  and the mapping $(E,\nabla)\mapsto (E,D)$ defines a functor.
  We will call $(E,D)$ \emph{the} semisimplification of $(E,\nabla)$.
\end{corollary}

Moreover, using the isomorphism of categories $\hCDz\simeq \hCHz$, we can describe a canonical
deformation to the semisimplification.

\begin{corollary}\label{deformationtosemisimplification}
  In the same context as in \Cref{semisimplification}, for each $h\in\C$
  there exists a unique $\eta_h$ such that $[\eta_h]=h[\eta]$ in $\HD^1(\End{E})$
  and $\left((E,D),\eta_h\right)$ is an object of $\hCDz$.
  For each $h$, the mapping $(E,\nabla)\mapsto (E,\nabla_h=D+\eta_h)$
  is a functor, compatible with tensor products, duals and pullbacks.

  Moreover, if $d$ is the minimal number of extensions necessary to construct $(E,\nabla)$ from semisimple flat bundles,
  the map $h\mapsto \nabla_h=D+\eta_h$ is polynomial of degree $\leq d$.
\end{corollary}

\begin{remark}
  Note that $\nabla_1=\nabla$ while $\nabla_0=D$ is the semisimplification.
\end{remark}

\begin{proof}
  Let us first show that $\left((E,D),h[\eta]\right)$ is an object of $\hCHz$.
  As $\eta$ is an extension, there exists a smooth decomposition $E=E_0\oplus\dotsb\oplus E_d$ of $E$
  preserved by $D$
  such that $\eta$ belongs to $\bigoplus_{i<j}\cA^1(\Hom{E_j}{E_i})$. 
  Hence $h[\eta]$ belongs to $\bigoplus_{i<j}\HD^1(\Hom{E_j}{E_i})$. Moreover, $[\eta]^2=0$ implies
  $(h[\eta])^2=0$. Hence $\left(E,D,h[\eta]\right)$ satisfies all the conditions to be an object of $\hCHz$
  (see \Cref{morphisms}).

  As objects in $\hCHz$ are in bijection with objects in $\hCDz$, there exists a unique $\eta_h$
  such that $\left(E,D,\eta_h\right)$ is an object of $\hCDz$ mapping to $\left(E,D,h[\eta]\right)$.
  This $\eta_h$ is thus characterized by $[\eta_h]=h[\eta]$.

  It remains to prove that for $d$ as in the statement, $h\mapsto D+\eta_h$ is polynomial of degree $\leq d$.
  We can do this by induction on $d$. If $d=0$, $\nabla$ is semisimple and hence $\eta_h=0$ for all $h$.
  Assume $d\geq 1$ and consider the decomposition $E=E_0\oplus\dotsb\oplus E_d$ as above.
  Then $D+\eta_h$ is of the form
  \begin{equation*}
    \begin{pmatrix}
      D_0    & \eta_{0,1,h} & \dotsb & \eta_{0,d-1,h} & \eta_{0,d,h}   \\
      0      & D_1                & \ddots &                      & \eta_{1,d,h}   \\
      \vdots & \ddots             & \ddots & \ddots               & \vdots               \\
      \vdots &                    & \ddots & D_{d-1}              & \eta_{d-1,d,h} \\
      0      & \dotsb             & \dotsb & 0                    & D_d
    \end{pmatrix}.
  \end{equation*}
  By the induction hypothesis, for $(i,j)\neq (0,d)$, $\eta_{i,j,h}$ is a polynomial of degree $\leq j-i$ in $h$.
  We want to show that $\eta_{0,d,h}$ is polynomial of degree $\leq d$ in $h$.
  As $D\eta + \eta^2=0$ and $D'\eta=0$, for each $h$, $\eta_{0,d,h}$
  satisfies the \emph{affine} system of equations
  \begin{equation}\label{equationeta}
    D\eta_{0,d,h}= -\sum_{i=1}^{d-1} \eta_{0,i,h}\eta_{i,d,h},\;D'\eta_{0,d,h}=0\text{ and } [\eta_{0,d,h}]=h[\eta_{0,d,1}].
  \end{equation}
  Now, $-\sum_{i=1}^{d-1} \eta_{0,i,h}\eta_{i,d,h}$ is a polynomial $\sum_{k=0}^d\xi_kh^k$ of degree $\leq d$ in $h$.
  By existence and uniqueness of $\eta_h$, we know that for each $h$, (\ref{equationeta}) has a unique solution $\eta_{0,d,h}$.
  Hence, by interpolating as $h$ varies in (\ref{equationeta}), we have that for each $k=0,\dotsc,d$,
  the system of equations
  \begin{equation*}
    Da= \xi_k,\;D'a=0\text{ and }[a] = \begin{cases}
       [\eta_{0,d,1}] &\text{ if }k=1,\\ 0 &\text{ if }k\neq 1
    \end{cases}
  \end{equation*}
  has a unique solution, that we denote $a_k$. Then $\sum_{k=0}^da_kh^k$ is a solution to (\ref{equationeta}),
  and by uniqueness, it must be $\eta_{0,d,h}$. Hence the result.
\end{proof}


\subsection{The case of smooth proper DM stacks}


In this subsection, we prove \Cref{semisimplification,deformationtosemisimplification} when the base $X$ is a smooth proper DM stack,
using the methods of \cite{simpsonLocalSystemsProper2011}. Although the non-Abelian Hodge correspondence extends to smooth proper DM stacks,
the formality results are less clear (see \cite[8.6]{simpsonLocalSystemsProper2011}). The difficulty is that the $DD'$-Lemma
is proved via Kähler identities, but not every smooth proper DM stack is Kähler. Although we could restrict to the Kähler case
for the purposes of this paper, we choose to give the proof in the more general case. Instead of trying to prove formality, we use a workaround to just deduce
the smooth proper DM stack case of \Cref{semisimplification,deformationtosemisimplification} from the case of smooth projective varieties.

All stacks considered will be smooth.
Note that by a DM stack we mean algebraic Deligne-Mumford stack over $\C$. These have realizations as smooth analytic and differentiable Deligne-Mumford stacks,
and hence we may talk about holomorphic and $\Cinf$-bundles over them. See \cite[2]{eyssidieuxInstantonsFramedSheaves2018} for more on this.

As mentioned above, the decomposition $D=D'+D''$ for $(E,D)$ semisimple over $X$ smooth proper DM stack exists, is unique and enjoys
the same compatibilities as in the case where $X$ is a compact Kähler manifold \cite[9.7]{simpsonLocalSystemsProper2011}.
Note that compatibility with respect to pullbacks is for algebraic maps $Y\ra X$.

\begin{proposition}\label{semisimplificationstacks}
  Any flat connection $\nabla$ on a $\Cinf$ bundle $E$ on a smooth proper DM stack $X$
  has a unique decomposition $\nabla=D+\eta$ with $D$ flat semisimple, $\eta$ iterated extension, and $D'\eta=0$.
  This decomposition is compatible with taking tensor products, duals and pullbacks,
  and the mapping $(E,\nabla)\mapsto (E,D)$ defines a functor.

  For each $h\in\C$ there exists a unique $\eta_h\in\cA^1(\End{E})$ such that for each algebraic map $p:Y\ra X$ with $Y$ smooth projective variety,
  $p^*(\eta_h)$ coincides with $(p^*\eta)_h$ as defined in \Cref{deformationtosemisimplification}.
  For each $h$, the mapping $(E,\nabla)\mapsto (E,\nabla_h=D+\eta_h)$
  is a functor, compatible with tensor products, duals and pullbacks.

  Moreover, if $d$ is the minimal number of extensions necessary to construct $(E,\nabla)$ from semisimple flat bundles,
  the map $h\mapsto \nabla_h=D+\eta_h$ is polynomial of degree $\leq d$.
\end{proposition}

\begin{proof}
  We use the results from \cite[§5]{simpsonLocalSystemsProper2011}. There, a hypercovering $Z_\bullet\ra X$
  by a simplicial scheme $Z_\bullet$ satisfying nice properties is constructed. Each $Z_i$ is a smooth projective variety.
  If $Z'\subset Z_0$ is the locus where $q:Z_0\ra X$ is étale, then $Z'\ra X$ is surjective. Moreover, the hypercovering satisfies
  descent for smooth vector bundles and natural structures (connections, differential forms) on smooth vector bundles.

  Then $(E,\nabla)$ corresponds to a bundle $(F,\nabla^F)$ on $Z_\bullet$ together with descent data.
  As each $Z_k$ is compact Kähler, by \Cref{semisimplification}, $\nabla^F$ has a unique decomposition $\nabla^F=D^F+\eta^F$
  with the desired properties.
  This decomposition is functorial and thus compatible with the descent data. Applying descent gives the desired decomposition
  $\nabla=D+\eta$ and its uniqueness. The semisimplicity of $D$ is a consequence of
  \cite[Lemma 8.2]{simpsonLocalSystemsProper2011}, and the compatibilities follow from uniqueness.
  
  Similarly, for each $h\in \C$, descent provides $\eta_h\in \cA^1(\End{E})$ such that its pullback to $Z_0$ is $(q^*\eta)_h$
  as defined in \Cref{deformationtosemisimplification}. Let $p:Y\ra X$ be an algebraic map from a smooth projective variety.
  Then $S=Y\times_X Z_0$ is a projective variety, potentially singular, but contains the open set $S'=Y\times_X Z'$
  which is étale surjective over $Y$. We may choose a resolution of singularities $\tilde{S}$ of $S$ which is an isomorphism over $S'\subset S$.
  We now have a diagram
  \[\begin{tikzcd}
    {S'} & {\tilde{S}} && {Z_0} \\
    \\
    & Y && {X.}
    \arrow[hook, from=1-1, to=1-2]
    \arrow["{\text{ét. surj.}}"', from=1-1, to=3-2]
    \arrow["{p'}", from=1-2, to=1-4]
    \arrow["{q'}"', from=1-2, to=3-2]
    \arrow["q"', from=1-4, to=3-4]
    \arrow["\alpha", shift left, shorten <=22pt, shorten >=22pt, Rightarrow, from=3-2, to=1-4]
    \arrow["p", from=3-2, to=3-4]
  \end{tikzcd}\]
  As $\tilde{S}\ra Y$ is étale surjective on an open subscheme, the map $\cA^1(p^*\End{E})\ra\cA^1(q^{\prime *}p^*\End{E})$ is injective.
  Now $p^*(\eta_h)$ and $(p^*\eta)_h$ both pull back to $(p^{\prime*}q^*\eta)_h$ on $\tilde{S}$ under the identification
  $\alpha^*:p^{\prime*}q^*\End{E}\simeq q^{\prime*}p^*\End{E}$. Hence they are equal, as desired. Again, the compatibilities follow from uniqueness.
\end{proof}


\subsection{Sketch of a direct proof of the semisimplification}


In this subsection, we briefly sketch a more direct proof of existence and uniqueness of the decomposition in \Cref{semisimplification}.

For existence, assume that $(E,\nabla)$ is a flat bundle over $X$ compact Kähler that is an extension of a semisimple connection by another semisimple connection.
Then there exists a smooth decomposition $E=E_1\oplus E_2$ such that $\nabla$ takes the form
\begin{equation*}
  \begin{pmatrix}
    D_1 & \eta \\
    0   & D_2
  \end{pmatrix}
\end{equation*}
where $D_i$ is a semisimple connection on $E_i$, $i=1,2$, and $\eta\in\cA^1(\Hom{E_2}{E_1})$ is a $1$-form,
which is closed for the connection $D$ on $\Hom{E_2}{E_1}$ induced by $D_1$ and $D_2$.
As $\mathrm{H}^1(\cA^\bullet(\Hom{E_2}{E_1}),D)$ coincides with $\mathrm{H}^1(\mathrm{Z}_{D'}^\bullet(\Hom{E_2}{E_1}),D)$ (\Cref{zigzag}),
there is a $D$- and $D'$-closed $1$-form $\xi\in\cA^1(\Hom{E_2}{E_1})$ representing $[\eta]$.

Now, the isomorphism class of the extension given by $\eta$ depends only on the cohomology class $[\eta]$. Hence $(E,\nabla)$ is isomorphic to $(E,\nabla')$
with
\begin{equation*}
  \nabla'=
  \begin{pmatrix}
    D_1 & \xi \\
    0   & D_2
  \end{pmatrix}.
\end{equation*}
The connection $\nabla'$ clearly admits the desired decomposition. Pulling it back along an isomorphism $f:(E,\nabla)\ra(E,\nabla')$
gives the desired decomposition of $\nabla$. Existence in the general case can be proved using similar arguments by induction on the number of extensions necessary
to construct $\nabla$ from semisimple connections.

Let us now turn to uniqueness. To prove uniqueness, we need only prove the following Lemma.

\begin{lemma*}
  Let $(E,\nabla^E=D^E+\eta^E)$ and $(F,\nabla^F=D^F+\eta^F)$ be two bundles with flat connections on $X$ compact Kähler
  with decompositions satisfying \Cref{semisimplification}. Then any map $f:(E,\nabla^E)\ra (F,\nabla^F)$
  is $D'$-closed for the operator D' induced by $D^E$ and $D^F$.
\end{lemma*}

Indeed, given two decompositions $\nabla=D+\eta=\tilde{D}+\tilde{\eta}$, according to the Lemma, the map $\id_E:(E,D+\eta)\ra(E,\tilde{D}+\tilde{\eta})$
must be $D'$-closed. But $D'$-closed implies flat in degree $0$ (this follows from \Cref{zigzag}).
Hence $\id_E:(E,D)\ra (E,\tilde{D})$ is a flat map, i.e. $D=\tilde{D}$.

Let us now explain the proof of the Lemma in a simple case. Assume that $D^F+\eta^F$ is of the form
\begin{equation*}
  \begin{pmatrix}
    D_1 & \eta \\
    0   & D_2
  \end{pmatrix}
\end{equation*}
for a smooth decomposition $F=F_1\oplus F_2$, and that $\nabla^E$ is semisimple.
Then, we have a map of long exact sequences in cohomology
\[\begin{tikzcd}
	0 & {\mathrm{H}^0(\cA^\bullet(\Hom{E}{F_1}))} & {\mathrm{H}^0(\cA^\bullet(\Hom{E}{F}))} & \cdots \\
	0 & {\mathrm{H}^0(\mathrm{Z}_{D'}^\bullet(\Hom{E}{F_1}))} & {\mathrm{H}^0(\mathrm{Z}_{D'}^\bullet(\Hom{E}{F_1}))} & \dotsb \\
	\cdots & {\mathrm{H}^0(\cA^\bullet(\Hom{E}{F_2}))} & {\mathrm{H}^1(\cA^\bullet(\Hom{E}{F_1}))} \\
	\dotsb & {\mathrm{H}^0(\mathrm{Z}_{D'}^\bullet(\Hom{E}{F_1}))} & {\mathrm{H}^0(\mathrm{Z}_{D'}^\bullet(\Hom{E}{F_1})).}
	\arrow[from=1-1, to=1-2]
	\arrow[from=1-2, to=1-3]
	\arrow[from=1-3, to=1-4]
	\arrow[from=2-1, to=2-2]
	\arrow["{i_1^0}", from=2-2, to=1-2]
	\arrow[from=2-2, to=2-3]
	\arrow["{i^0}", from=2-3, to=1-3]
	\arrow[from=2-3, to=2-4]
	\arrow[from=3-1, to=3-2]
	\arrow[from=3-2, to=3-3]
	\arrow[from=4-1, to=4-2]
	\arrow["{i_2^0}", from=4-2, to=3-2]
	\arrow[from=4-2, to=4-3]
	\arrow["{i_1^1}", from=4-3, to=3-3]
\end{tikzcd}\]
Because the induced connections on $\Hom{E}{F_1}$ and $\Hom{E}{F_2}$ are semisimple, by \Cref{zigzag},
we see that $i^0_1$, $i^0_2$ and $i^1_1$ are isomorphisms. Hence, by the $5$-lemma, so is $i^0$.
Notice that a map $f:(E,\nabla^E)\ra(F,\nabla^F)$ is an element of $\mathrm{H}^0(\cA^\bullet(\Hom{E}{F}))$.
As $i^0$ is an isomorphism, $f$ must be $D'$-closed. This proves the Lemma in this case.

The case where $\nabla^E$ is semisimple can then be dealt with by induction on the number of extensions necessary to construct $\nabla^F$
from semisimple connections, using the $5$-lemma at each induction step. The general case is then deduced by induction
on the number of extensions necessary to define $\nabla^E$, using again the $5$-lemma at each induction step.


\section{Ocneanu rigidity}\label{sectionOcneanu}


The main ingredient in the proof of our main result (\Cref{mainresult}) is Ocneanu rigidity, which we now state, and a textbook account of which is available in
\cite[chp. 9.1]{etingofTensorCategories2015}.

\begin{theorem}[{Ocneanu rigidity \cite[2.28]{etingofFusionCategories2005}}]\label{Ocneanurigidity}
  A fusion category does not admit nontrivial infinitesimal deformations
  (i.e. its associator admits no nontrivial infinitesimal deformation).
  In particular, the number of such categories up to equivalence with a given Grothendieck ring is finite.
\end{theorem}

As a corollary of \Cref{Ocneanurigidity}, we get the following.

\begin{corollary}\label{corollarydeformations}
  Let $C$ be a ribbon or braided category over $\C$.
  Then for any continuous family of ribbon or braided fusion categories $(C_t)_{t\in X}$, with $C_0=C$, $X$ path-connected,
  and where only the associators vary, $C_t$ is isomorphic to $C$ for all $t$.
\end{corollary}
\begin{proof}
  By \cite[7.7]{godfardHodgeStructuresConformal2024}
  (which is a corollary of \Cref{Ocneanurigidity}), the isomorphism class of $C_t$ is constant
  along paths in $X$. As $X$ is path-connected, the isomorphism class of $C_t$ is independent of $t$.
\end{proof}


\section{Proof of the main result}\label{sectionproof}


In this section, $\Nu$ is a modular, genus-zero modular or braided functor over $\C$ with
associated modular, genus-zero modular or braided fusion category $\CN$.

As in \Cref{semisimplificationstacks}, for any flat connection $\nabla$ over a smooth proper DM stack, we will denote by
$(\nabla_h)_{h\in\C}$ the canonical deformation to the semisimplification. The deformation is polynomial
and $\nabla_1=\nabla$ while $\nabla_0$ is semisimple. As this deformation is functorial, for each $h\in \C$,
the $(\Nu_g(\ul),\nabla_h)$ together with the gluing, forgetful, and permutation isomorphism of $\Nu$
form a geometric functor, that we denote $\Nu_h$. This geometric functor has an associated category $\mathrm{C}_{\Nu_h}$.
Note that by the definition of the functor $\Nu\mapsto \CN$, $\mathrm{C}_{\Nu_h}$ has the same underlying braiding (and twist) as $\CN$,
and only the associator may vary. Moreover, as each $\nabla_h$ is a polynomial in $h$, its monodromy varies continuously in $h$.
We can thus apply \Cref{corollarydeformations} to the family $(\mathrm{C}_{\Nu_h})_{h\in\C}$. Hence $\mathrm{C}_{\Nu_0}$ is isomorphic to $\CN$.
By full faithfulness in \Cref{fullfaithfulnesses}, this implies that $\Nu_0$ is isomorphic to $\Nu$.
In particular, for each $g$ and $\ul$, $(\Nu_g(\ul),\nabla)$ is isomorphic to $(\Nu_g(\ul),\nabla_0)$, which is semisimple.


\bibliographystyle{plain}
\bibliography{biblio}

\end{document}

%% file: tikzstyles.tikzstyles
\tikzstyle{green}=[text={black!30!green}]
\tikzstyle{blue}=[text=blue]
\tikzstyle{red}=[text=red]
\tikzstyle{puncture}=[fill=white, draw=red, shape=circle, minimum size=1pt]
\tikzstyle{blackpuncture}=[fill=white, draw=black, shape=circle, minimum size=1pt]
\tikzstyle{cyan}=[text=cyan]

\tikzstyle{markings}=[-, draw=red, line width=1pt, line cap=round]
\tikzstyle{overbraid}=[-, draw=white, fill=none, line width=6pt]
\tikzstyle{thick}=[-, line width=2pt, draw=blue]
\tikzstyle{dashedline}=[-, dashed]
\tikzstyle{dottedline}=[-, dash pattern=on 0.75pt off 0.75pt, line width=0.75pt]
\tikzstyle{thin red}=[-, line width=0.25pt, draw=black]
\tikzstyle{tangle}=[-, draw=blue, line width=1pt, fill={blue!20}]
\tikzstyle{scc}=[-, draw={black!30!green}, fill={blue!20}, line width=1pt]
\tikzstyle{inner boundary}=[-, fill=white]
\tikzstyle{outer boundary}=[-, fill={red!20}]
\tikzstyle{lowerboundery}=[-, line width=1.5pt, line cap=round, draw=red]
\tikzstyle{upperboundery}=[-, line width=1.5pt, line cap=round, draw=blue]
\tikzstyle{dottedcycle}=[-, draw=blue, dash pattern=on 0.5pt off 1pt on 4pt off 1pt, decoration={markings, mark=at position 0.5 with {\arrow{>}}}, postaction=decorate]
\tikzstyle{cycle}=[-, draw=blue, decoration={markings, mark=at position 0.5 with {\arrow{>}}}, postaction=decorate]
\tikzstyle{path}=[-, draw=cyan, line width=0.25pt]
\tikzstyle{arrowpath}=[-, draw=cyan, line width=0.25pt, decoration={markings, mark=at position 0.5 with {\arrow{>}}}, postaction=decorate]
\tikzstyle{orientedpath}=[-, line width=0.25pt, decoration={markings, mark=at position 0.5 with {\arrow{<}}}, postaction=decorate]
\tikzstyle{inner square}=[-, fill={blue!20}]
\tikzstyle{outer square}=[-, fill={red!20}]
\tikzstyle{blueline}=[-, draw=blue]
\tikzstyle{greenline}=[-, draw=green]
\tikzstyle{bluesquare}=[-, draw=blue, fill={blue!20}]

%% file: figures/path.pdf_tex
\begingroup%
  \makeatletter%
  \providecommand\color[2][]{%
    \errmessage{(Inkscape) Color is used for the text in Inkscape, but the package 'color.sty' is not loaded}%
    \renewcommand\color[2][]{}%
  }%
  \providecommand\transparent[1]{%
    \errmessage{(Inkscape) Transparency is used (non-zero) for the text in Inkscape, but the package 'transparent.sty' is not loaded}%
    \renewcommand\transparent[1]{}%
  }%
  \providecommand\rotatebox[2]{#2}%
  \newcommand*\fsize{\dimexpr\f@size pt\relax}%
  \newcommand*\lineheight[1]{\fontsize{\fsize}{#1\fsize}\selectfont}%
  \ifx\svgwidth\undefined%
    \setlength{\unitlength}{252.14098833bp}%
    \ifx\svgscale\undefined%
      \relax%
    \else%
      \setlength{\unitlength}{\unitlength * \real{\svgscale}}%
    \fi%
  \else%
    \setlength{\unitlength}{\svgwidth}%
  \fi%
  \global\let\svgwidth\undefined%
  \global\let\svgscale\undefined%
  \makeatother%
  \begin{picture}(1,0.8474042)%
    \lineheight{1}%
    \setlength\tabcolsep{0pt}%
    \put(0,0){\includegraphics[width=\unitlength,page=1]{path.pdf}}%
    \put(0.4520528,0.01846593){\color[rgb]{0.75294118,0.25098039,0.25098039}\makebox(0,0)[lt]{\lineheight{1.25}\smash{\begin{tabular}[t]{l}$\mu$\end{tabular}}}}%
    \put(0.18553488,0.43568534){\color[rgb]{0.25098039,0.25098039,0.75294118}\makebox(0,0)[lt]{\lineheight{1.25}\smash{\begin{tabular}[t]{l}$\lambda_1$\end{tabular}}}}%
    \put(0.32632602,0.43982864){\color[rgb]{0.25098039,0.25098039,0.75294118}\makebox(0,0)[lt]{\lineheight{1.25}\smash{\begin{tabular}[t]{l}$\lambda_2$\end{tabular}}}}%
    \put(0.73960939,0.43986913){\color[rgb]{0.25098039,0.25098039,0.75294118}\makebox(0,0)[lt]{\lineheight{1.25}\smash{\begin{tabular}[t]{l}$\lambda_3$\end{tabular}}}}%
    \put(0,0){\includegraphics[width=\unitlength,page=2]{path.pdf}}%
  \end{picture}%
\endgroup%